\pdfoutput=1
\documentclass{elsarticle}
\usepackage[cmex10]{amsmath}
\usepackage{amsthm,amssymb,amsfonts}
\usepackage{algorithmic,algorithm}
\usepackage{tikz}
\usepackage{xcolor}
\usepackage{array}
\usepackage[caption=false,font=normalsize,labelfont=sf,textfont=sf]{subfig}
\usepackage{caption}
\usepackage{longtable}
\usepackage{multirow}
\usepackage{makecell}
\usepackage{diagbox}
\usepackage{textcomp}
\usepackage{stfloats}
\usepackage{url}
\usepackage{verbatim}
\usepackage{graphicx}
\usepackage{arydshln} 
\usepackage[utf8]{inputenc}
\usepackage[T1]{fontenc}
\usepackage{ifthen}

\newtheorem{theorem}{\bf Theorem}
\newtheorem{lemma}[theorem]{\bf Lemma}

\newtheorem{definition}{\bf Definition}

\date{}
\makeatletter
\def\ps@pprintTitle{%
  \let\@oddhead\@empty
  \let\@evenhead\@empty
  \def\@oddfoot{\reset@font\hfil\thepage\hfil}
  \let\@evenfoot\@oddfoot
}
\makeatother
\begin{document}

\begin{frontmatter}

\title{Minimum size of insertion/deletion/substitution balls} %% Article title

\author[label1,label2]{Yuhang Pi}
\ead{piyuhang@amss.ac.cn}
\author[label1,label2]{Zhifang Zhang}
\ead{zfz@amss.ac.cn}

\address[label1]{Academy of Mathematics and Systems Science, 
Chinese Academy of Sciences, Beijing 100190, China}
\address[label2]{School of Mathematical Sciences, 
University of Chinese Academy of Sciences, Beijing 100049, China}

\begin{abstract}
Let $n,q,t,s,p$ be non-negative integers where $n\geq s$ and $q\geq 1$. 
For $\mathbf{x}\in A_{q}^{n}\triangleq\{ 0,1,\ldots,q-1 \}^{n}$, 
let the $t$-insertion $s$-deletion $p$-substitution ball of $\mathbf{x}$, 
denoted by $\mathcal{B}_{t,s,p}(\mathbf{x})$, be the set of sequences 
in $A_{q}^{n+t-s}$ which can be obtained from $\mathbf{x}$ 
by performing $t$ insertions, $s$ deletions, and at most $p$ substitutions. 
We establish that for any $\mathbf{x}\in A_{q}^{n}$, 
$|\mathcal{B}_{t,s,p}(\mathbf{x})|\geq\sum_{i=0}^{t+p}\binom{n+t-s}{i}(q-1)^{i}$, 
with equality holding if and only if $t=s=0\vee s=p=0\vee s+p\geq n\vee r(\mathbf{x})=1$. 
Here, $r(\mathbf{x})$ denotes the number of runs in $\mathbf{x}$, and 
a run in $\mathbf{x}$ is a maximum continuous subsequence of identical symbols.
\end{abstract}


\end{frontmatter}
%-----------------------------------------------------------------------------------------------%

%-----------------------------------------------------------------------------------------------%

%-----------------------------------------------------------------------------------------------%

%-----------------------------------------------------------------------------------------------%

%-----------------------------------------------------------------------------------------------%
\section{Introduction}\label{sec1}

The research of insertion/deletion/substitution balls contains $4$ fundamental questions: 
size, maximum size, minimum size, and average size. 
Considering general $t$-insertion $s$-deletion $p$-substitution balls 
$\mathcal{B}_{t,s,p}(\mathbf{x})$, as shown in Table \ref{tab1}, 
the determined results remain limited. 

\begin{table}[htbp]
\caption{For $\mathbf{x}\in A_{q}^{n}$, 
determined values related to $\mathcal{B}_{t,s,p}(\mathbf{x})$.}
\label{tab1}
\begin{center}
\begin{tabular}{|c|c|}
\hline
\diagbox[width=0.3cm,height=0.3cm]{}{} & 
\raisebox{0pt}[1em][0.6em]{\footnotesize Determined values for parameter $(t,s,p)_{q}$} 
\\
\hline
\multirow{2}{*}{\footnotesize $|\mathcal{B}_{t,s,p}(\mathbf{x})|$} & 
\footnotesize $(t,0,0)_{q}$ \cite{Lev.1974,Lev.1992}; $(0,0,p)_{q}$; $(1,1,0)_{q}$ \cite{Sal.2013}; 
$(0,1,p)_{2},(1,0,1)_{2}$ \cite{Abu-Sini.2021};
\\
& \footnotesize $(0,1,1)_{q}$ \cite{Sma.2023}; $(2,1,0)_{q},(0,1,2)_{q},(1,0,1)_{q}$ \cite{Abb.2024};
\\
\hline
\multirow{2}{*}{\raisebox{0pt}[1em][0.6em]{
\footnotesize 
$\underset{\mathbf{x}\in A_{q}^{n}}{\max}(|\mathcal{B}_{t,s,p}(\mathbf{x})|)$}} & 
\multirow{2}{*}{
\footnotesize $(t,0,0)_{q}$ \cite{Lev.1974,Lev.1992}; $(0,s,0)_{q}$ \cite{Hir.2000}; 
$(0,0,p)_{q}$; $(1,1,0)_{q}$ \cite{Bar-Lev.2023}; $(0,1,p)_{2},(1,0,1)_{2}$ \cite{Abu-Sini.2021};}
\\
&
\\
\hline
\multirow{2}{*}{\raisebox{0pt}[1em][0.6em]{
\footnotesize 
$\underset{\mathbf{x}\in A_{q}^{n}}{\min}(|\mathcal{B}_{t,s,p}(\mathbf{x})|)$}} & 
\multirow{2}{*}{
\footnotesize $(t,0,0)_{q}$ \cite{Lev.1974,Lev.1992}; $(0,s,0)_{q}$; 
$(0,0,p)_{q}$; $(t,t,0)_{q}$ \cite{Bar-Lev.2023}; $(0,1,p)_{2},(1,0,1)_{2}$ \cite{Abu-Sini.2021};}
\\
&
\\
\hline
\multirow{2}{*}{
\raisebox{0pt}[1.2em][0.6em]{\footnotesize 
$\underset{\mathbf{x}\in A_{q}^{n}}{\mathbb{E}}(|\mathcal{B}_{t,s,p}(\mathbf{x})|)$}} & 
\multirow{2}{*}{
\footnotesize $(t,0,0)_{q}$ \cite{Lev.1974,Lev.1992}; $(0,s,0)_{q}$(cf., \cite{Hir.2000,Lev.2002}); 
$(0,0,p)_{q}$; $(1,1,0)_{q}$ \cite{Bar-Lev.2023};}
\\
&
\\
\hline
\end{tabular}
\end{center}
\end{table}

Throughout this paper, we assume that $n,q,t,s,p$ are non-negative integers, 
$n\geq s$ and $q\geq 1$. The following theorem establishes the minimum size of 
$t$-insertion $s$-deletion $p$-substitution balls, along with the necessary and sufficient 
condition for achieving this minimum.

\begin{theorem}\label{M.t1.1}
For any $\mathbf{x}\in A_{q}^{n}$, 
\begin{equation*}
|\mathcal{B}_{t,s,p}(\mathbf{x})|\geq\sum_{i=0}^{t+p}\binom{n+t-s}{i}(q-1)^{i}, 
\end{equation*}
with equality holding if and only if $t=s=0\vee s=p=0\vee s+p\geq n\vee r(\mathbf{x})=1$.
\end{theorem}

The rest of this paper is organized as follows. 
We introduce notations and known conclusions in Section \ref{sec2}, 
and sketch the proof of Theorem \ref{M.t1.1} in Section \ref{sec3}. 
To complete the proof, we discuss the three cases $t=s$, $t<s$, and $t>s$ 
in Sections \ref{sec4}--\ref{sec6}, respectively. 
%-----------------------------------------------------------------------------------------------%

%-----------------------------------------------------------------------------------------------%

%-----------------------------------------------------------------------------------------------%

%-----------------------------------------------------------------------------------------------%

%-----------------------------------------------------------------------------------------------%
\section{Preliminaries}\label{sec2}

We adopt the following conventions: when enumerating elements of a finite set of integers, 
they are listed in ascending order, e.g., $T=\{ i_{1},\ldots,i_{n} \}$ then $i_{1}<\cdots<i_{n}$. 
If the subscript of a letter is $0$, the corresponding index is set to $0$, e.g., $i_{0}=0$. 
Let $[n]=\{ 1,2,\ldots,n \}$, $A_{q}=\{ 0,\ldots, q-1 \}$, 
$A_{q}^{n}=\{ x_{1}\cdots x_{n}\mid x_{1},\ldots,x_{n}\in A_{q} \}$, 
and $[l,m]=\{ l,l+1,\ldots,m \}$ for integers $l$ and $m$. 
For $a,b\in A_{q}$, define $a\oplus b,a\ominus b\in A_{q}$ 
such that $a\oplus b\equiv a+b\bmod{q}$ and $a\ominus b\equiv a-b\bmod{q}$. 
Particularly, define $\overline{a}=a\oplus 1$. Note that if $q\geq 2$, $\overline{a}\neq a$. 
For $\mathbf{x}=x_{1}\cdots x_{n}\in A_{q}^{n}$, $r(\mathbf{x})$ denotes the number of runs 
in $\mathbf{x}$, where a run in $\mathbf{x}$ is a maximum continuous subsequence of 
identical symbols. For $T=\{ i_{1},\ldots,i_{j} \}\subseteq [n]$, 
$\mathbf{x}_{T}$ denotes $x_{i_{1}}\cdots x_{i_{j}}$. For $\mathbf{y}\in A_{q}^{n}$, 
$H(\mathbf{x},\mathbf{y})$ denotes the Hamming distance between $\mathbf{x}$ and $\mathbf{y}$. 
If $S$ is a set and $m$ is a non-negative integer, 
$\binom{S}{m}$ denotes $\{ S'\subseteq S\mid |S'|=m \}$. 
We define the insertion/deletion/substitution balls as follows. 

\begin{definition}
The $t$-insertion $s$-deletion $p$-substitution ball of $\mathbf{x}\in A_{q}^{n}$ is defined by 
\begin{equation*}
\resizebox{0.9\columnwidth}{!}{$
\mathcal{B}_{t,s,p}(\mathbf{x})=
\left\{ \mathbf{z}\in A_{q}^{n+t-s}\mid \exists S_{1}\in\binom{[n]}{n-s}, S_{2}\in\binom{[n+t-s]}{n-s}, 
H(\mathbf{x}_{S_{1}},\mathbf{z}_{S_{2}})\leq p \right\},$}
\end{equation*} 
i.e., the set of sequences in $A_{q}^{n+t-s}$ 
with $t$ insertions, $s$ deletions, and at most $p$ substitutions of $\mathbf{x}$. 
\end{definition}

For example, for $\mathbf{x}\!=\!0000\in\! A_{2}^{4}$, 
$\mathcal{B}_{1,1,1}(\mathbf{x})\!=\!A_{2}^{4}\setminus\{ 1111,1110,1101,1011,0111 \}$. 
We simplify the $t$-insertion ball $\mathcal{B}_{t,0,0}(\mathbf{x})$ to $\mathcal{I}_{t}(\mathbf{x})$ and 
$s$-deletion ball $\mathcal{B}_{0,s,0}(\mathbf{x})$ to $\mathcal{D}_{s}(\mathbf{x})$. 
If $\mathbf{z}\in\mathcal{D}_{s}(\mathbf{x})$, 
the matching set $M(\mathbf{z},\mathbf{x})$ is defined as $\{ i_{1},\ldots,i_{n-s} \}\subseteq [n]$, 
such that $i_{1}=\min\{ e\in [n]\mid x_{e}=z_{1} \}$ 
and for $j\in [2,n-s]$, $i_{j}=\min\{ e\in [i_{j-1}+1,n]\mid x_{e}=z_{j} \}$. 
For example, $M(110,21010)=\{ 2,4,5 \}$. 

Next, we present some known size of balls. 

For any $\mathbf{x}\in A_{q}^{n}$, the size of $p$-substitution ball is known to be 
\begin{equation}\label{M.e2.1}
|\mathcal{B}_{0,0,p}(\mathbf{x})|=\sum_{i=0}^{p}\binom{n}{i}(q-1)^{i}.
\end{equation}
As shown in Theorem \ref{M.t1.1}, since $r(0^{n})=1$, we provide the size of 
$\mathcal{B}_{t,s,p}(0^{n})$ in the following Lemma \ref{M.l2.1} and aim to demonstrate that 
it always achieves the minimum. 
\begin{lemma}\label{M.l2.1}
For $0^{n}\in A_{q}^{n}$, 
\begin{equation*}
|\mathcal{B}_{t,s,p}(0^{n})|=\sum_{i=0}^{t+p}\binom{n+t-s}{i}(q-1)^{i}.
\end{equation*}
\end{lemma}

\begin{proof}
If $p\geq n-s$, then $\mathcal{B}_{0,s,p}(0^{n})=A_{q}^{n-s}$, 
$\mathcal{B}_{t,s,p}(0^{n})=A_{q}^{n+t-s}$, and thus 
\begin{equation*}
|\mathcal{B}_{t,s,p}(0^{n})|=q^{n+t-s}
=\sum_{i=0}^{n+t-s}\binom{n+t-s}{i}(q-1)^{i}
=\sum_{i=0}^{t+p}\binom{n+t-s}{i}(q-1)^{i}.
\end{equation*}

If $p<n-s$, then 
\begin{align*}
\mathcal{B}_{t,s,p}(0^{n})
=&\{ \mathbf{y}\in A_{q}^{n+t-s}\mid 
\text{the number of } 0\text{s} \text{ in } \mathbf{y}\in [n-s-p,n+t-s] \}
\\
=&\{ \mathbf{y}\in A_{q}^{n+t-s}\mid 
\text{the number of nonzero symbols} \text{ in } \mathbf{y}\in [0,t+p] \},
\end{align*}
which implies 
\begin{equation*}
|\mathcal{B}_{t,s,p}(0^{n})|=\sum_{i=0}^{t+p}\binom{n+t-s}{i}(q-1)^{i}.
\end{equation*}
This completes the proof. 
\end{proof}

Levenshtein \cite{Lev.1974,Lev.1992} determined the size of $t$-insertion balls 
by induction (cf., \cite{Lev.2001b}). We present a new proof in Lemma \ref{M.l2.2}, 
as it inspires the method in Lemma \ref{M.l6.1}. 
\begin{lemma}\label{M.l2.2}
For any $\mathbf{x}\in A_{q}^{n}$, 
\begin{equation*}
|\mathcal{I}_{t}(\mathbf{x})|=\sum_{i=0}^{t}\binom{n+t}{i}(q-1)^{i}.
\end{equation*}
\end{lemma}

\begin{proof}
For any $\mathbf{x}\in A_{q}^{n}$, 
we construct a bijection from $\mathcal{I}_{t}(0^{n})$ to $\mathcal{I}_{t}(\mathbf{x})$, 
thereby completing the proof through Lemma \ref{M.l2.1}. 

Firstly, we construct a mapping 
\begin{align*}
f: \mathcal{I}_{t}(0^{n})\rightarrow&\;\mathcal{I}_{t}(\mathbf{x}),
\\
\mathbf{y}\rightarrow&\;\mathbf{z}.
\end{align*}
To achieve this, 
for $\mathbf{y}\in \mathcal{I}_{t}(0^{n})$ (i.e., $0^{n}\in \mathcal{D}_{t}(\mathbf{y})$), we assume 
\begin{equation*}
I\triangleq M(0^{n},\mathbf{y})=\{ i_{1},\ldots,i_{n} \}\subseteq [n+t].
\end{equation*}
Based on our initial assumption in this section, 
$i_{0}=0<i_{1}<\cdots<i_{n}\leq n+t$. We define 
$f(\mathbf{y})=\mathbf{z}\in A_{q}^{n+t}$, 
where 
\begin{equation*}
z_{e}=
\begin{cases}
y_{e}\oplus x_{l}, 
& \text{ if } e\in [i_{n}] \text{ (and thus } \exists!\; l\in [n] \text{ such that } e\in [i_{l-1}+1,i_{l}]), 
\\
y_{e}, 
& \text{ if } e\in [i_{n}+1,n+t].
\end{cases}
\end{equation*}
For any $l\in [n]$, since $i_{l}\in [i_{l-1}+1,i_{l}]$ and $y_{i_{l}}=0$ by the definition of 
$M(0^{n},\mathbf{y})$, we have $z_{i_{l}}=y_{i_{l}}\oplus x_{l}=x_{l}$ 
and thus $\mathbf{z}_{I}=\mathbf{x}$, which implies 
$\mathbf{z}=f(\mathbf{y})\in \mathcal{I}_{t}(\mathbf{x})$. 
Therefore, the defined $f$ is indeed a mapping from 
$\mathcal{I}_{t}(0^{n})$ to $\mathcal{I}_{t}(\mathbf{x})$. 

Before proceeding with the subsequent proof, let us first examine an example. 
Consider the above mapping from $\mathcal{I}_{4}(0000)$ to $\mathcal{I}_{4}(\mathbf{x})$ where 
$\mathbf{x}=1001\in A_{3}^{4}$. 
For $\mathbf{y}=20100100\in \mathcal{I}_{4}(0000)$, 
$I=M(0^{4},\mathbf{y})=\{ 2,4,5,7 \}$. 
Assuming $f(\mathbf{y})=\mathbf{z}$, we obtain 
\begin{equation*}
z_{e}=
\begin{cases}
y_{e}\oplus 1, & \text{ if } e\in [1,2],
\\
y_{e}\oplus 0, & \text{ if } e\in [3,4],
\\
y_{e}\oplus 0, & \text{ if } e\in [5,5],
\\
y_{e}\oplus 1, & \text{ if } e\in [6,7],
\\
y_{e}, & \text{ if } e\in [8,8],
\end{cases}
\end{equation*}
and thus $\mathbf{z}=01100210\in \mathcal{I}_{4}(1001)$.

Secondly, we prove that the defined mapping $f$ is injective. 
On the contrary, there exist $\mathbf{y}\neq \mathbf{y}'\in \mathcal{I}_{t}(0^{n})$ 
such that $f(\mathbf{y})=\mathbf{z}=f(\mathbf{y}')=\mathbf{z}'$. 
We similarly define $I'=M(0^{n},\mathbf{y}')=\{ i_{1}',\ldots,i_{n}' \}$. 
Note that $\mathbf{y}\neq \mathbf{y}'$ implies $q\geq 2$. 

\textit{Case 1:} 
$I\neq I'$. Let $c=\min\{ e\in [n]\mid i_{e}\neq i_{e}' \}$. 
As shown in Fig. \ref{M.f2.1}, without loss of generality, we may assume $i_{c}<i_{c}'$. 
By the definition of $M(0^{n},\mathbf{y}')$, $y_{i_{c}}'\neq 0$ 
(Otherwise $i_{c}\in M(0^{n},\mathbf{y}')$, a contradiction). 
Combining with $i_{c}\in [i_{c-1}+1,i_{c}]\cap [i_{c-1}'+1,i_{c}']$, we obtain 
$z_{i_{c}}=y_{i_{c}}\oplus x_{c}=x_{c}$ and $z_{i_{c}}'=y_{i_{c}}'\oplus x_{c}\neq x_{c}$, 
which contradicts to $\mathbf{z}=\mathbf{z}'$. 

\begin{figure}[ht]
\centering
\begin{tikzpicture}
\node at (0,2.5) {\small Position:};

\node at (1.5,2.5) {\small $\cdots$};

\node at (2.5,2.5) {\small $i_{1}$};

\node at (3.5,2.5) {\small $\cdots$};

\node at (4.5,2.5) {\small $i_{2}$};

\node at (5.5,2.5) {\small $\cdots$};

\node at (6.5,2.5) {\small $i_{c-1}$};

\node at (7.5,2.5) {\small $\cdots$};

\node at (8.5,2.5) {\small $i_{c}$};

\node at (9.5,2.5) {\small $\cdots$};

\node at (10.5,2.5) {\small $i_{c}'$};

\node at (11.2,2.5) {\small $\cdots$};
%-----------------------------------------------------------------------------------------------%

%-----------------------------------------------------------------------------------------------%

%-----------------------------------------------------------------------------------------------%
\node at (0.5,1.5) {$\mathbf{y}=$};

\node at (1.5,1.5) {$\ast\ast\ast$};

\node at (2.5,1.5) {$0$};

\node at (3.5,1.5) {$\ast\ast\ast$};

\node at (4.5,1.5) {$0$};

\node at (5.5,1.5) {$\ast\ast\ast$};

\node at (6.5,1.5) {$0$};

\node at (7.5,1.5) {$\ast\ast\ast$};

\node at (8.5,1.5) {$0$};

\node at (9.5,1.5) {$\cdots$};
%-----------------------------------------------------------------------------------------------%

%-----------------------------------------------------------------------------------------------%

%-----------------------------------------------------------------------------------------------%
\node at (0.5,0.5) {$\mathbf{y}'=$};

\node at (1.5,0.5) {$\ast\ast\ast$};

\node at (2.5,0.5) {$0$};

\node at (3.5,0.5) {$\ast\ast\ast$};

\node at (4.5,0.5) {$0$};

\node at (5.5,0.5) {$\ast\ast\ast$};

\node at (6.5,0.5) {$0$};

\node at (7.5,0.5) {$\ast\ast\ast$};

\node at (8.5,0.5) {$\ast$};

\node at (9.5,0.5) {$\ast\ast\ast$};

\node at (10.5,0.5) {$0$};

\node at (11.2,0.5) {$\cdots$};

\end{tikzpicture}
\caption{\textit{Case 1}, where $\ast$ and $\ast\ast\ast$ represent 
an element in $A_{q}\setminus\{ 0 \}$ and a sequence over $A_{q}\setminus\{ 0 \}$, 
respectively.}
\label{M.f2.1}
\end{figure}

\textit{Case 2:} 
$I=I'$. Since $\mathbf{y}\neq \mathbf{y}'$, there exists $m\in [n+t]$ such that $y_{m}\neq y_{m}'$. 

\textit{Case 2.1:} 
$m\in [i_{n}]$. Then $\exists !$ $l\in [n]$ such that $m\in [i_{l-1}+1,i_{l}]$. 
Noting that $z_{m}=y_{m}\oplus x_{l}\neq y_{m}'\oplus x_{l}=z_{m}'$, 
we obtain $\mathbf{z}\neq \mathbf{z}'$, a contradiction. 

\textit{Case 2.2:} 
$m\in [i_{n}+1,n+t]$. Then $z_{m}=y_{m}\neq y_{m}'=z_{m}'$, which implies 
$\mathbf{z}\neq \mathbf{z}'$, a contradiction. 

After examining all cases, we obtain a contradiction. 
Consequently, the defined mapping $f$ is indeed an injection. 

Thirdly, we prove that the defined mapping $f$ is surjective. 
For $\mathbf{z}\in \mathcal{I}_{t}(\mathbf{x})$, we assume 
$M(\mathbf{x},\mathbf{z})=\{ u_{1},\ldots,u_{n} \}\subseteq [n+t]$. 
Define $\mathbf{y}\in A_{q}^{n+t}$, where 
\begin{equation*}
y_{e}=
\begin{cases}
z_{e}\ominus x_{l}, 
& \text{ if } e\in [u_{n}] \text{ (and thus } \exists!\; l\in [n] \text{ such that } e\in [u_{l-1}+1,u_{l}]), 
\\
z_{e}, 
& \text{ if } e\in [u_{n}+1,n+t].
\end{cases}
\end{equation*}
By the definition of $M(\mathbf{x},\mathbf{z})$, for any $l\in [n]$ and $m\in [u_{l-1}+1,u_{l}-1]$, 
$z_{u_{l}}=x_{l}$ and $z_{m}\neq x_{l}$. 
Correspondingly, $y_{u_{l}}=z_{u_{l}}\ominus x_{l}=0$ and $y_{m}=z_{m}\ominus x_{l}\neq 0$. 
It follows that $\mathbf{y}\in \mathcal{I}_{t}(0^{n})$ and $M(0^{n},\mathbf{y})=\{ u_{1},\ldots,u_{n} \}$. 
Furthermore, it is straightforward to verify $f(\mathbf{y})=\mathbf{z}$ 
which implies that the defined mapping $f$ is indeed a surjection. 

Combining these results, we conclude that 
$f$ is a bijection from $\mathcal{I}_{t}(0^{n})$ to $\mathcal{I}_{t}(\mathbf{x})$. 
This completes the proof through Lemma \ref{M.l2.1}. 
\end{proof}
%-----------------------------------------------------------------------------------------------%

%-----------------------------------------------------------------------------------------------%

%-----------------------------------------------------------------------------------------------%

%-----------------------------------------------------------------------------------------------%

%-----------------------------------------------------------------------------------------------%
\section{A Proof Sketch of Theorem \ref{M.t1.1}}\label{sec3}

\begin{proof}
By the equation in \eqref{M.e2.1}, Lemma \ref{M.l2.1}, and Lemma \ref{M.l2.2}, 
we can easily verify if $t=s=0\vee s=p=0\vee s+p\geq n\vee r(\mathbf{x})=1$, 
then 
\begin{equation*}
|\mathcal{B}_{t,s,p}(\mathbf{x})|=\sum_{i=0}^{t+p}\binom{n+t-s}{i}(q-1)^{i}. 
\end{equation*}
To complete the proof, in the rest of this paper, we always assume 
\begin{equation*}
(t,s)\neq (0,0)\wedge (s,p)\neq (0,0)\wedge s+p<n\wedge r(\mathbf{x})>1
\end{equation*}
and proceed to prove 
\begin{equation}\label{M.e3.1}
|\mathcal{B}_{t,s,p}(\mathbf{x})|
>\sum_{i=0}^{t+p}\binom{n+t-s}{i}(q-1)^{i}.
\end{equation}
Note that $r(\mathbf{x})>1$ implies $q\geq 2$. 

Specifically, we examine the three cases $t=s$, $t<s$, and $t>s$ in Sections \ref{sec4}--\ref{sec6}, 
and complete the final proof for each case in Lemma \ref{M.l4.1}, Lemma \ref{M.l5.2}, 
and Lemma \ref{M.l6.2}, respectively. 
\end{proof}
%-----------------------------------------------------------------------------------------------%

%-----------------------------------------------------------------------------------------------%

%-----------------------------------------------------------------------------------------------%

%-----------------------------------------------------------------------------------------------%

%-----------------------------------------------------------------------------------------------%
\section{$t=s$}\label{sec4}

We slightly modify the method in \cite{Bar-Lev.2023} to complete the proof. 

\begin{lemma}\label{M.l4.1}
For $\mathbf{x}\in A_{q}^{n}$, if 
\begin{equation*}
(t,s)\neq (0,0)\wedge (s,p)\neq (0,0)\wedge s+p<n\wedge r(\mathbf{x})>1\wedge t=s,
\end{equation*} 
then 
\begin{equation*}
|\mathcal{B}_{t,s,p}(\mathbf{x})|
>\sum_{i=0}^{t+p}\binom{n+t-s}{i}(q-1)^{i}.
\end{equation*}
\end{lemma}

\begin{proof}
We simplify the problem to proving 
\begin{equation*}
|\mathcal{B}_{t,t,p}(\mathbf{x})|
>\sum_{i=0}^{t+p}\binom{n}{i}(q-1)^{i}
\end{equation*}
under the condition 
$t>0\wedge t+p<n\wedge r(\mathbf{x})>1$. 

On the one hand, noting that a substitution can be replaced by an insertion followed by a deletion, 
$\mathcal{B}_{t,t,p}(\mathbf{x})\supseteq \mathcal{B}_{0,0,t+p}(\mathbf{x})$. 
Combining with \eqref{M.e2.1}, we obtain 
\begin{equation}\label{M.e4.1}
|\mathcal{B}_{t,t,p}(\mathbf{x})|
\geq |\mathcal{B}_{0,0,t+p}(\mathbf{x})|
=\sum_{i=0}^{t+p}\binom{n}{i}(q-1)^{i}.
\end{equation}

On the other hand, noting that $r(\mathbf{x})>1$, 
there exists $m\in [n-1]$ such that $x_{m}\neq x_{m+1}$. 
Since $0\leq t+p-1\leq n-2$, we can arbitrarily choose a set 
$E\in\binom{[n]\setminus\{ m,m+1 \}}{t+p-1}$ and define $\mathbf{z}\in A_{q}^{n}$, where 
\begin{equation*}
z_{e}=
\begin{cases}
x_{e}, & \text{ if } e\notin E\cup\{ m,m+1 \},
\\
x_{m+1}, & \text{ if } e=m,
\\
x_{m}, & \text{ if } e=m+1,
\\
\overline{x_{e}}, & \text{ if } e\in E.
\end{cases} 
\end{equation*}
Clearly, $\mathbf{z}\in\mathcal{B}_{1,1,t+p-1}(\mathbf{x})\subseteq \mathcal{B}_{t,t,p}(\mathbf{x})$, 
since $\mathbf{z}$ can be obtained from $\mathbf{x}$ by 
substituting $x_{e}$ with $\overline{x_{e}}$ for all $e\in E$, deleting $x_{m}$, 
and inserting it after $x_{m+1}$. 
Also, $\mathbf{z}\notin \mathcal{B}_{0,0,t+p}(\mathbf{x})$ since $H(\mathbf{z},\mathbf{x})=t+p+1$. 
Therefore, $\mathbf{z}\in \mathcal{B}_{t,t,p}(\mathbf{x})\setminus\mathcal{B}_{0,0,t+p}(\mathbf{x})$, 
which implies the strict inequality in \eqref{M.e4.1} holds and completes the proof. 
\end{proof}
%-----------------------------------------------------------------------------------------------%

%-----------------------------------------------------------------------------------------------%

%-----------------------------------------------------------------------------------------------%

%-----------------------------------------------------------------------------------------------%

%-----------------------------------------------------------------------------------------------%
\section{$t<s$}\label{sec5}

The maximum intersection size of two distinct substitution balls is given in Lemma \ref{Lev.l5.1}, 
which will be employed in Lemma \ref{M.l5.1}. 

\begin{lemma}[\cite{Lev.2001a}]\label{Lev.l5.1}
\begin{equation*}
\max_{\mathbf{x}\neq \mathbf{y}\in A_{q}^{n}}
|\mathcal{B}_{0,0,p}(\mathbf{x})\cap \mathcal{B}_{0,0,p}(\mathbf{y})|
=q\sum_{i=0}^{p-1}\binom{n-1}{i}(q-1)^{i}.
\end{equation*}
\end{lemma}

We first prove \eqref{M.e3.1} for $t=0$. 

\begin{lemma}\label{M.l5.1}
For $\mathbf{x}\in A_{q}^{n}$, if 
\begin{equation*}
(t,s)\neq (0,0)\wedge (s,p)\neq (0,0)\wedge s+p<n\wedge r(\mathbf{x})>1\wedge t<s\wedge t=0,
\end{equation*} 
then 
\begin{equation*}
|\mathcal{B}_{t,s,p}(\mathbf{x})|
>\sum_{i=0}^{t+p}\binom{n+t-s}{i}(q-1)^{i}.
\end{equation*}
\end{lemma}

\begin{proof}
We simplify the problem to proving 
\begin{equation*}
|\mathcal{B}_{0,s,p}(\mathbf{x})|
>\sum_{i=0}^{p}\binom{n-s}{i}(q-1)^{i}
\end{equation*}
under the condition 
$s>0\wedge s+p<n\wedge r(\mathbf{x})>1$, 
which is directly obtained by the following three claims. 

\textit{Claim 1:} 
There exist $\mathbf{u}\neq\mathbf{v}\in \mathcal{D}_{s}(\mathbf{x})$. 

\textit{Claim 2:} 
$|\mathcal{B}_{0,0,p}(\mathbf{u})|=|\mathcal{B}_{0,0,p}(\mathbf{v})|=
\sum_{i=0}^{p}\binom{n-s}{i}(q-1)^{i}$. 

\textit{Claim 3:} 
$|\mathcal{B}_{0,0,p}(\mathbf{u})\cap \mathcal{B}_{0,0,p}(\mathbf{v})|
<\sum_{i=0}^{p}\binom{n-s}{i}(q-1)^{i}$. 

\textit{Proof of Claim 1:} 
Since $r(\mathbf{x})>1$, there exists $m\in [n-1]$ such that $x_{m}\neq x_{m+1}$. 
Noting that $0\leq n-s-1\leq n-2$, we can arbitrarily choose a set 
$E\in\binom{[n]\setminus\{ m,m+1 \}}{n-s-1}$ 
and define $\mathbf{u}=\mathbf{x}_{E\cup\{ m \}},\mathbf{v}=\mathbf{x}_{E\cup\{ m+1 \}}$. 
Clearly, $\mathbf{u}\neq\mathbf{v}\in \mathcal{D}_{s}(\mathbf{x})$. 
Thus, \textit{Claim 1} is proved. 

\textit{Proof of Claim 2:} It is directly obtained by the equation in \eqref{M.e2.1}. 

\textit{Proof of Claim 3:} By Lemma \ref{Lev.l5.1}, 
\begin{align}
|\mathcal{B}_{0,0,p}(\mathbf{u})\cap \mathcal{B}_{0,0,p}(\mathbf{v})|
\leq& q\sum_{i=0}^{p-1}\binom{n-s-1}{i}(q-1)^{i}\notag
\\
=& \sum_{i=0}^{p-1}\binom{n-s-1}{i}(q-1)^{i+1}+\sum_{i=0}^{p-1}\binom{n-s-1}{i}(q-1)^{i}\notag
\\
=& \sum_{i=1}^{p}\binom{n-s-1}{i-1}(q-1)^{i}+\sum_{i=0}^{p-1}\binom{n-s-1}{i}(q-1)^{i}\notag
\\
<& \sum_{i=0}^{p}\binom{n-s-1}{i-1}(q-1)^{i}+\sum_{i=0}^{p}\binom{n-s-1}{i}(q-1)^{i}\label{M.e5.1}
\\
=& \sum_{i=0}^{p}\binom{n-s}{i}(q-1)^{i},\notag
\end{align}
where the strict inequality in \eqref{M.e5.1} holds since $n-s-1\geq p$, and 
$\binom{n-s-1}{p}(q-1)^{p}>0$. Thus, \textit{Claim 3} is proved. 
\end{proof}

Next, we return to proving \eqref{M.e3.1} for general $t$. 

\begin{lemma}\label{M.l5.2}
For $\mathbf{x}\in A_{q}^{n}$, if 
\begin{equation*}
(t,s)\neq (0,0)\wedge (s,p)\neq (0,0)\wedge s+p<n\wedge r(\mathbf{x})>1\wedge t<s,
\end{equation*} 
then 
\begin{equation*}
|\mathcal{B}_{t,s,p}(\mathbf{x})|
>\sum_{i=0}^{t+p}\binom{n+t-s}{i}(q-1)^{i}.
\end{equation*}
\end{lemma}

\begin{proof}
Noting that 
\begin{equation*}
(t,s)\neq (0,0)\wedge (s,p)\neq (0,0)\wedge s+p<n\wedge r(\mathbf{x})>1\wedge t<s
\end{equation*} 
implies
\begin{equation*}
\resizebox{\columnwidth}{!}{$
(0,s-t)\neq (0,0)\wedge (s-t,t+p)\neq (0,0)\wedge (s-t)+(t+p)<n\wedge r(\mathbf{x})>1
\wedge 0<s-t,$}
\end{equation*} 
we can employ Lemma \ref{M.l5.1} for $|\mathcal{B}_{0,s-t,t+p}(\mathbf{x})|$ 
and thus 
\begin{equation*}
|\mathcal{B}_{t,s,p}(\mathbf{x})|
\geq|\mathcal{B}_{0,s-t,t+p}(\mathbf{x})|
>\sum_{i=0}^{t+p}\binom{n+t-s}{i}(q-1)^{i},
\end{equation*}
which completes the proof. 
\end{proof}
%-----------------------------------------------------------------------------------------------%

%-----------------------------------------------------------------------------------------------%

%-----------------------------------------------------------------------------------------------%

%-----------------------------------------------------------------------------------------------%

%-----------------------------------------------------------------------------------------------%
\section{$t>s$}\label{sec6}

From \eqref{M.e2.1} and Lemma \ref{M.l2.2}, 
we can observe that the size of insertion balls and substitution balls is uniform, 
independent from the specific choice of $\mathbf{x}\in A_{q}^{n}$. 
Intriguingly, as we will demonstrate in this section, 
generally the size of insertion/substitution balls is no longer uniform. 

We first prove \eqref{M.e3.1} for $s=0$. 

\begin{lemma}\label{M.l6.1}
For $\mathbf{x}\in A_{q}^{n}$, if 
\begin{equation*}
(t,s)\neq (0,0)\wedge (s,p)\neq (0,0)\wedge s+p<n\wedge r(\mathbf{x})>1\wedge t>s\wedge s=0,
\end{equation*} 
then 
\begin{equation*}
|\mathcal{B}_{t,s,p}(\mathbf{x})|
>\sum_{i=0}^{t+p}\binom{n+t-s}{i}(q-1)^{i}.
\end{equation*}
\end{lemma}

\begin{proof}
We simplify the problem to proving 
\begin{equation*}
|\mathcal{B}_{t,0,p}(\mathbf{x})|
>\sum_{i=0}^{t+p}\binom{n+t}{i}(q-1)^{i}
\end{equation*}
under the condition 
$t>0\wedge p>0\wedge p<n\wedge r(\mathbf{x})>1$. 
Inspired by the new proof of Lemma \ref{M.l2.2}, 
we construct a mapping from 
$\mathcal{B}_{t,0,p}(0^{n})$ to $\mathcal{B}_{t,0,p}(\mathbf{x})$ 
that is injective but not surjective to complete the proof. 

Firstly, we construct a mapping 
\begin{align*}
f: \mathcal{B}_{t,0,p}(0^{n})\rightarrow&\;\mathcal{B}_{t,0,p}(\mathbf{x}),
\\
\mathbf{y}\rightarrow&\;\mathbf{z}.
\end{align*}
To achieve this, for $\mathbf{y}\in \mathcal{B}_{t,0,p}(0^{n})$, 
as mentioned in the proof of Lemma \ref{M.l2.1}, 
the number of $0$s in $\mathbf{y}$ is at least $n-p$. 
Thus, we can assume 
\begin{equation*}
I\triangleq M(0^{n-p},\mathbf{y})=\{ i_{1},\ldots,i_{n-p} \}\subseteq [n+t].
\end{equation*}
Additionally, assume that the set $J=\{ j_{1},\ldots,j_{p} \}$ contains the smallest $p$ 
elements in $[n+t]\setminus I$. Let $K=I\cup J=\{ k_{1},\ldots,k_{n} \}$. 
Based on our initial assumption in Section \ref{sec2}, 
$k_{0}=0<k_{1}<\cdots<k_{n}\leq n+t$. We define 
$f(\mathbf{y})=\mathbf{z}\in A_{q}^{n+t}$, 
where 
\begin{equation*}
z_{e}=
\begin{cases}
y_{e}\oplus x_{l}, 
& \text{ if } e\in [k_{n}] \text{ (and thus } \exists!\; l\in [n] \text{ such that } e\in [k_{l-1}+1,k_{l}]), 
\\
y_{e}, 
& \text{ if } e\in [k_{n}+1,n+t].
\end{cases}
\end{equation*}
Since $\mathbf{y}_{I}=0^{n-p}$ and $I\subseteq K$, 
there are at most $p$ nonzero symbols in $\mathbf{y}_{K}$. 
Combining with $z_{k_{l}}=y_{k_{l}}\oplus x_{l}$ for any $l\in [n]$, 
$\mathbf{z}_{K}$ and $\mathbf{x}$ differ in at most $p$ positions, 
i.e., $H(\mathbf{z}_{K},\mathbf{x})\leq p$, 
which implies $\mathbf{z}=f(\mathbf{y})\in \mathcal{B}_{t,0,p}(\mathbf{x})$. 
Therefore, the defined $f$ is indeed a mapping from 
$\mathcal{B}_{t,0,p}(0^{n})$ to $\mathcal{B}_{t,0,p}(\mathbf{x})$. 

Before proceeding with the subsequent proof, let us provide some additional explanations 
and examine an example. When $p=0$, the mapping $f$ defined here coincides with that 
in Lemma \ref{M.l2.2}. However, given that the current condition is $0<p<n$, it will be seen that 
although the mapping remains injective, it is no longer surjective. 
Consider the above mapping from $\mathcal{B}_{4,0,2}(0000)$ to 
$\mathcal{B}_{4,0,2}(\mathbf{x})$ where $\mathbf{x}=1001\in A_{3}^{4}$. 
Take $\mathbf{y}=20100100\in \mathcal{B}_{4,0,2}(0000)$, with 
$I=M(0^{2},\mathbf{y})=\{ 2,4 \}$, $J=\{ 1,3 \}$, and $K=\{ 1,2,3,4 \}$. 
Assuming $f(\mathbf{y})=\mathbf{z}$, we obtain 
\begin{equation*}
z_{e}=
\begin{cases}
y_{e}\oplus 1, & \text{ for } e\in [1,1],
\\
y_{e}\oplus 0, & \text{ for } e\in [2,2],
\\
y_{e}\oplus 0, & \text{ for } e\in [3,3],
\\
y_{e}\oplus 1, & \text{ for } e\in [4,4],
\\
y_{e}, & \text{ for } e\in [5,8],
\end{cases}
\end{equation*}
and thus $\mathbf{z}=00110100\in \mathcal{B}_{4,0,2}(1001)$.

Secondly, we prove that the defined mapping $f$ is injective. 
On the contrary, there exist $\mathbf{y}\neq \mathbf{y}'\in \mathcal{B}_{t,0,p}(0^{n})$ 
such that $f(\mathbf{y})=\mathbf{z}=f(\mathbf{y}')=\mathbf{z}'$. 
We similarly define $I'=M(0^{n-p},\mathbf{y}')=\{ i_{1}',\ldots,i_{n-p}' \}$, 
$J'=\{ j_{1}',\ldots,j_{p}' \}$ which contains the smallest $p$ 
elements in $[n+t]\setminus I'$, and $K'=I'\cup J'=\{ k_{1}',\ldots,k_{n}' \}$. 

\textit{Case 1:} 
$I\neq I'$. Let $c=\min\{ e\in [n-p]\mid i_{e}\neq i_{e}' \}$. 
As shown in Fig. \ref{M.f6.1}, without loss of generality, we may assume $i_{c}<i_{c}'$. 
By the definition of $M(0^{n-p},\mathbf{y}')$, $y_{i_{c}}'\neq 0$ 
(Otherwise $i_{c}\in M(0^{n-p},\mathbf{y}')$, a contradiction). 

\begin{figure}[ht]
\centering
\begin{tikzpicture}
\node at (0,2.5) {\small Position:};

\node at (1.5,2.5) {\small $\cdots$};

\node at (2.5,2.5) {\small $i_{1}$};

\node at (3.5,2.5) {\small $\cdots$};

\node at (4.5,2.5) {\small $i_{2}$};

\node at (5.5,2.5) {\small $\cdots$};

\node at (6.5,2.5) {\small $i_{c-1}$};

\node at (7.5,2.5) {\small $\cdots$};

\node at (8.5,2.5) {\small $i_{c}$};

\node at (9.5,2.5) {\small $\cdots$};

\node at (10.5,2.5) {\small $i_{c}'$};

\node at (11.2,2.5) {\small $\cdots$};
%-----------------------------------------------------------------------------------------------%

%-----------------------------------------------------------------------------------------------%

%-----------------------------------------------------------------------------------------------%
\node at (0.5,1.5) {$\mathbf{y}=$};

\node at (1.5,1.5) {$\ast\ast\ast$};

\node at (2.5,1.5) {$0$};

\node at (3.5,1.5) {$\ast\ast\ast$};

\node at (4.5,1.5) {$0$};

\node at (5.5,1.5) {$\ast\ast\ast$};

\node at (6.5,1.5) {$0$};

\node at (7.5,1.5) {$\ast\ast\ast$};

\node at (8.5,1.5) {$0$};

\node at (9.5,1.5) {$\cdots$};
%-----------------------------------------------------------------------------------------------%

%-----------------------------------------------------------------------------------------------%

%-----------------------------------------------------------------------------------------------%
\node at (0.5,0.5) {$\mathbf{y}'=$};

\node at (1.5,0.5) {$\ast\ast\ast$};

\node at (2.5,0.5) {$0$};

\node at (3.5,0.5) {$\ast\ast\ast$};

\node at (4.5,0.5) {$0$};

\node at (5.5,0.5) {$\ast\ast\ast$};

\node at (6.5,0.5) {$0$};

\node at (7.5,0.5) {$\ast\ast\ast$};

\node at (8.5,0.5) {$\ast$};

\node at (9.5,0.5) {$\ast\ast\ast$};

\node at (10.5,0.5) {$0$};

\node at (11.2,0.5) {$\cdots$};

\end{tikzpicture}
\caption{\textit{Case 1}, where $\ast$ and $\ast\ast\ast$ represent 
an element in $A_{q}\setminus\{ 0 \}$ and a sequence over $A_{q}\setminus\{ 0 \}$, 
respectively.}
\label{M.f6.1}
\end{figure}

\textit{Case 1.1:} $j_{p}<i_{c}$. 
This indicates that $|[i_{c}-1]\setminus I|\geq p$. 
Referring to Fig. \ref{M.f6.1}, $|[i_{c}-1]\setminus I'|\geq p$ also holds. 
By the definition of $J$ and $J'$, we obtain $J=J'\subseteq [i_{c}-1]$ which implies 
$i_{c}=k_{c+p}$, $i_{c}'=k_{c+p}'$, and $i_{c}\in [k_{c+p-1}'+1,k_{c+p}']$. 
Thus, $z_{i_{c}}=y_{i_{c}}\oplus x_{c+p}=x_{c+p}$ and 
$z_{i_{c}}'=y_{i_{c}}'\oplus x_{c+p}\neq x_{c+p}$, 
which contradicts to $\mathbf{z}=\mathbf{z}'$. 

\textit{Case 1.2:} $j_{p}>i_{c}$. 
This indicates that $|[i_{c}-1]\setminus I|<p$. 
Referring to Fig. \ref{M.f6.1}, $|[i_{c}-1]\setminus I'|<p$ also holds. 
By the definition of $J$ and $J'$, we obtain $[i_{c}]\setminus I\subseteq J$ and 
$[i_{c}]\setminus I'\subseteq J'$ which implies $[i_{c}]\subseteq K\cap K'$. 
Then, we have $k_{1}=k_{1}'=1,\ldots,k_{i_{c}}=k_{i_{c}}'=i_{c}$. 
Thus, $z_{i_{c}}=y_{i_{c}}\oplus x_{i_{c}}=x_{i_{c}}$ and 
$z_{i_{c}}'=y_{i_{c}}'\oplus x_{i_{c}}\neq x_{i_{c}}$, 
which contradicts to $\mathbf{z}=\mathbf{z}'$. 

\textit{Case 2:} 
$I=I'$. By definition, $J=J'$ and $K=K'=\{ k_{1},\ldots,k_{n} \}$. 
Since $\mathbf{y}\neq \mathbf{y}'$, there exists $m\in [n+t]$ such that $y_{m}\neq y_{m}'$. 

\textit{Case 2.1:} 
$m\in [k_{n}]$. Then $\exists !$ $l\in [n]$ such that $m\in [k_{l-1}+1,k_{l}]$. 
Noting that $z_{m}=y_{m}\oplus x_{l}\neq y_{m}'\oplus x_{l}=z_{m}'$, 
we obtain $\mathbf{z}\neq \mathbf{z}'$, a contradiction. 

\textit{Case 2.2:} 
$m\in [k_{n}+1,n+t]$. Then $z_{m}=y_{m}\neq y_{m}'=z_{m}'$, which implies 
$\mathbf{z}\neq \mathbf{z}'$, a contradiction. 

After examining all cases, we obtain a contradiction. 
Consequently, the defined mapping $f$ is indeed an injection. 

Thirdly, we prove that the defined mapping $f$ is not surjective, 
which is directly obtained from the following two claims. 

\textit{Claim 1:} 
There exists $\mathbf{z}\in\mathcal{B}_{t,0,p}(\mathbf{x})$ such that $z_{i}\neq x_{i}$ 
for any $i\in [p]$ and $\mathbf{x}_{[p+1,n]}\notin \mathcal{D}_{t}(\mathbf{z}_{[p+1,n+t]})$. 

\textit{Claim 2:} 
Such $\mathbf{z}$ in \textit{Claim 1} has no preimage under the mapping $f$. 

\textit{Proof of Claim 1:} 
We obtain $\mathbf{z}$ from $\mathbf{x}$ in two steps. 

In the first stage, as shown in Fig. \ref{M.f6.2}, we substitute $x_{n}$ with 
$\overline{x_{n}}$ and insert a prefix $\overline{x_{1}}$ and a suffix $(\overline{x_{n}})^{t-1}$ 
from $\mathbf{x}$ to obtain an intermediate sequence $\mathbf{w}$. 
This process requires a total of $t$ insertions and one substitution. 

\begin{figure}[ht]
\centering
\begin{tikzpicture}
\node at (0,2.5) {\small Position:};

\node at (1.5,2.5) {\small $1$};

\node at (2.3,2.5) {\small $2$};

\node at (3.1,2.5) {\small $\cdots$};

\node at (3.9,2.5) {\small $p$};

\node at (4.7,2.53) {\small $p+1$};

\node at (5.5,2.5) {\small $\cdots$};

\node at (6.3,2.53) {\small $n-1$};

\node at (7.3,2.51) {\small $n$};

\node at (8.3,2.53) {\small $n+1$};

\node at (9.3,2.53) {\small $n+2$};

\node at (10.1,2.5) {\small $\cdots$};

\node at (10.9,2.53) {\small $n+t$};
%-----------------------------------------------------------------------------------------------%

%-----------------------------------------------------------------------------------------------%

%-----------------------------------------------------------------------------------------------%
\node at (0.5,1.5) {$\mathbf{x}=$};

\node at (1.5,1.5) {\small $x_{1}$};

\node at (2.3,1.5) {\small $x_{2}$};

\node at (3.1,1.5) {\small $\cdots$};

\node at (3.9,1.5) {\small $x_{p}$};

\node at (4.7,1.5) {\small $x_{p+1}$};

\node at (5.5,1.5) {\small $\cdots$};

\node at (6.3,1.5) {\small $x_{n-1}$};

\node at (7.3,1.5) {\small $x_{n}$};
%-----------------------------------------------------------------------------------------------%

%-----------------------------------------------------------------------------------------------%

%-----------------------------------------------------------------------------------------------%
\node at (0.5,0.5) {$\mathbf{w}=$};

\node at (1.5,0.53) {\small $\overline{x_{1}}$};

\node at (2.3,0.5) {\small $x_{1}$};

\node at (3.1,0.5) {\small $\cdots$};

\node at (3.9,0.5) {\small $x_{p-1}$};

\node at (4.7,0.5) {\small $x_{p}$};

\node at (5.5,0.5) {\small $\cdots$};

\node at (6.3,0.5) {\small $x_{n-2}$};

\node at (7.3,0.5) {\small $x_{n-1}$};

\node at (8.3,0.53) {\small $\overline{x_{n}}$};

\node at (9.3,0.53) {\small $\overline{x_{n}}$};

\node at (10.1,0.5) {\small $\cdots$};

\node at (10.9,0.53) {\small $\overline{x_{n}}$};

\end{tikzpicture}
\caption{In the first stage, we substitute $x_{n}$ with $\overline{x_{n}}$ and 
insert a prefix $\overline{x_{1}}$ and a suffix $(\overline{x_{n}})^{t-1}$ 
from $\mathbf{x}$ to obtain $\mathbf{w}$.}
\label{M.f6.2}
\end{figure}

In the second stage, noting that $r(\mathbf{x})>1$ implies $r(\mathbf{x}_{[p]})>1$ or 
$r(\mathbf{x}_{[p,n]})>1$, we discuss these two cases to obtain $\mathbf{z}$ from $\mathbf{w}$ 
by at most $p-1$ substitutions. The common point between them is that the positions 
of substitutions in $\mathbf{w}$ belong to the set $[2,n]$. 

\textit{Case 1':} 
$r(\mathbf{x}_{[p]})>1$. Then $p\geq 2$ and there exists $i\in [p-1]$ such that $x_{i}\neq x_{i+1}$. 
As shown in Fig. \ref{M.f6.2} and Fig. \ref{M.f6.3}, we perform substitutions on the symbols 
of $\mathbf{w}$ by substituting $w_{e}$ with $\overline{x_{e}}$ for all 
$e\in ([2,p]\cup\{ n \})\setminus\{ i+1 \}$. This yields the final sequence $\mathbf{z}$. 

\begin{figure}[ht]
\centering
\begin{tikzpicture}
\node at (0,2.5) {\small Position:};

\node at (1.5,2.5) {\small $1$};

\node at (2.3,2.5) {\small $2$};

\node at (3.1,2.5) {\small $\cdots$};

\node at (3.9,2.5) {\small $i+1$};

\node at (4.7,2.5) {\small $\cdots$};

\node at (5.5,2.5) {\small $p$};

\node at (6.3,2.53) {\small $p+1$};

\node at (7.1,2.5) {\small $\cdots$};

\node at (7.9,2.5) {\small $n-1$};

\node at (8.7,2.5) {\small $n$};

\node at (9.5,2.5) {\small $n+1$};

\node at (10.3,2.5) {\small $\cdots$};

\node at (11.1,2.5) {\small $n+t$};
%-----------------------------------------------------------------------------------------------%

%-----------------------------------------------------------------------------------------------%

%-----------------------------------------------------------------------------------------------%
\node at (0.5,1.5) {$\mathbf{x}=$};

\node at (1.5,1.5) {\small $x_{1}$};

\node at (2.3,1.5) {\small $x_{2}$};

\node at (3.1,1.5) {\small $\cdots$};

\node at (3.9,1.5) {\small $x_{i+1}$};

\node at (4.7,1.5) {\small $\cdots$};

\node at (5.5,1.5) {\small $x_{p}$};

\node at (6.3,1.5) {\small $x_{p+1}$};

\node at (7.1,1.5) {\small $\cdots$};

\node at (7.9,1.5) {\small $x_{n-1}$};

\node at (8.7,1.5) {\small $x_{n}$};
%-----------------------------------------------------------------------------------------------%

%-----------------------------------------------------------------------------------------------%

%-----------------------------------------------------------------------------------------------%
\node at (0.5,0.5) {$\mathbf{z}=$};

\node at (1.5,0.53) {\small $\overline{x_{1}}$};

\node at (2.3,0.53) {\small $\overline{x_{2}}$};

\node at (3.1,0.5) {\small $\cdots$};

\node at (3.9,0.5) {\small $x_{i}$};

\node at (4.7,0.5) {\small $\cdots$};

\node at (5.5,0.53) {\small $\overline{x_{p}}$};

\node at (6.3,0.5) {\small $x_{p}$};

\node at (7.1,0.5) {\small $\cdots$};

\node at (7.9,0.5) {\small $x_{n-2}$};

\node at (8.7,0.53) {\small $\overline{x_{n}}$};

\node at (9.5,0.53) {\small $\overline{x_{n}}$};

\node at (10.3,0.5) {\small $\cdots$};

\node at (11.1,0.53) {\small $\overline{x_{n}}$};

\node at (3.9,1) {\small $\neq $};

\end{tikzpicture}
\caption{For \textit{Case 1'} in the second stage, we substitute $w_{e}$ with $\overline{x_{e}}$ 
for all $e\in ([2,p]\cup\{ n \})\setminus\{ i+1 \}$ from $\mathbf{w}$ to obtain $\mathbf{z}$.}
\label{M.f6.3}
\end{figure}

\textit{Case 2':} 
$r(\mathbf{x}_{[p,n]})>1$. Then there exists $i\in [p,n-1]$ such that $x_{i}\neq x_{i+1}$. 
As shown in Fig. \ref{M.f6.2} and Fig. \ref{M.f6.4}, we perform substitutions on the symbols 
of $\mathbf{w}$ by substituting $w_{e}$ with $\overline{x_{e}}$ for all $e\in [2,p]$. 
This yields the final sequence $\mathbf{z}$. 

\begin{figure}[ht]
\centering
\begin{tikzpicture}
\node at (0,2.5) {\small Position:};

\node at (1.5,2.5) {\small $1$};

\node at (2.3,2.5) {\small $2$};

\node at (3.1,2.5) {\small $\cdots$};

\node at (3.9,2.5) {\small $p$};

\node at (4.7,2.52) {\small $p+1$};

\node at (5.5,2.5) {\small $\cdots$};

\node at (6.3,2.51) {\small $i+1$};

\node at (7.1,2.5) {\small $\cdots$};

\node at (7.9,2.52) {\small $n-1$};

\node at (8.7,2.51) {\small $n$};

\node at (9.5,2.52) {\small $n+1$};

\node at (10.3,2.5) {\small $\cdots$};

\node at (11.1,2.52) {\small $n+t$};
%-----------------------------------------------------------------------------------------------%

%-----------------------------------------------------------------------------------------------%

%-----------------------------------------------------------------------------------------------%
\node at (0.5,1.5) {$\mathbf{x}=$};

\node at (1.5,1.5) {\small $x_{1}$};

\node at (2.3,1.5) {\small $x_{2}$};

\node at (3.1,1.5) {\small $\cdots$};

\node at (3.9,1.5) {\small $x_{p}$};

\node at (4.7,1.5) {\small $x_{p+1}$};

\node at (5.5,1.5) {\small $\cdots$};

\node at (6.3,1.5) {\small $x_{i+1}$};

\node at (7.1,1.5) {\small $\cdots$};

\node at (7.9,1.5) {\small $x_{n-1}$};

\node at (8.7,1.5) {\small $x_{n}$};
%-----------------------------------------------------------------------------------------------%

%-----------------------------------------------------------------------------------------------%

%-----------------------------------------------------------------------------------------------%
\node at (0.5,0.5) {$\mathbf{z}=$};

\node at (1.5,0.53) {\small $\overline{x_{1}}$};

\node at (2.3,0.53) {\small $\overline{x_{2}}$};

\node at (3.1,0.5) {\small $\cdots$};

\node at (3.9,0.53) {\small $\overline{x_{p}}$};

\node at (4.7,0.5) {\small $x_{p}$};

\node at (5.5,0.5) {\small $\cdots$};

\node at (6.3,0.5) {\small $x_{i}$};

\node at (7.1,0.5) {\small $\cdots$};

\node at (7.9,0.5) {\small $x_{n-2}$};

\node at (8.7,0.5) {\small $x_{n-1}$};

\node at (9.5,0.53) {\small $\overline{x_{n}}$};

\node at (10.3,0.5) {\small $\cdots$};

\node at (11.1,0.53) {\small $\overline{x_{n}}$};

\node at (6.3,1) {\small $\neq$};

\end{tikzpicture}
\caption{For \textit{Case 2'} in the second stage, we substitute $w_{e}$ with $\overline{x_{e}}$ 
for all $e\in [2,p]$ from $\mathbf{w}$ to obtain $\mathbf{z}$.}
\label{M.f6.4}
\end{figure}

We can directly observe that the defined 
$\mathbf{z}\in\mathcal{B}_{0,0,p-1}(\mathbf{w})\subseteq \mathcal{B}_{t,0,p}(\mathbf{x})$ 
satisfies the requirements. Thus, \textit{Claim 1} is proved. 

\textit{Proof of Claim 2:} On the contrary, 
$\mathbf{z}\in\mathcal{B}_{t,0,p}(\mathbf{x})$ that satisfies the requirements of \textit{Claim 1} 
has a preimage $\mathbf{y}\in\mathcal{B}_{t,0,p}(0^{n})$ under $f$, i.e., $f(\mathbf{y})=\mathbf{z}$. 
Recall the sets $I$, $J$, and $K$ defined in the first part. 
For any $e\in [p]$, if $e\notin I$, then by the definition of $J$, $e\in J$ since 
$|[e-1]\setminus I|\leq p-1$. This indicates $[p]\subseteq I\cup J=K$. 
Thus, for any $e\in [p]$, $k_{e}=e$ and thus $y_{e}=z_{e}\ominus x_{e}\neq 0$ by the requirements 
in \textit{Claim 1}. It follows that $I=\{ i_{1},\ldots,i_{n-p} \}\subseteq [p+1,n+t]$, $J=[p]$, 
and thus $i_{e}=k_{e+p}$ for any $e\in [n-p]$. 
Noting that $y_{i_{e}}=0$ for $e\in [n-p]$, 
$z_{i_{e}}=y_{i_{e}}\oplus x_{e+p}=x_{e+p}$, which implies 
$\mathbf{z}_{I}=\mathbf{x}_{[p+1,n]}$. 
This contradicts to $\mathbf{x}_{[p+1,n]}\notin \mathcal{D}_{t}(\mathbf{z}_{[p+1,n+t]})$. 
Thus, \textit{Claim 2} is proved. 

Combining these results, we conclude that 
$f$ is an injection from $\mathcal{B}_{t,0,p}(0^{n})$ to $\mathcal{B}_{t,0,p}(\mathbf{x})$ 
but not a surjection. This completes the proof through Lemma \ref{M.l2.1}. 
\end{proof}

Next, we return to proving \eqref{M.e3.1} for general $s$. 

\begin{lemma}\label{M.l6.2}
For $\mathbf{x}\in A_{q}^{n}$, if 
\begin{equation*}
(t,s)\neq (0,0)\wedge (s,p)\neq (0,0)\wedge s+p<n\wedge r(\mathbf{x})>1\wedge t>s,
\end{equation*} 
then 
\begin{equation*}
|\mathcal{B}_{t,s,p}(\mathbf{x})|
>\sum_{i=0}^{t+p}\binom{n+t-s}{i}(q-1)^{i}.
\end{equation*}
\end{lemma}

\begin{proof}
Noting that 
\begin{equation*}
(t,s)\neq (0,0)\wedge (s,p)\neq (0,0)\wedge s+p<n\wedge r(\mathbf{x})>1\wedge t>s
\end{equation*} 
implies
\begin{equation*}
(t-s,0)\neq (0,0)\wedge (0,s+p)\neq (0,0)\wedge 0+(s+p)<n\wedge r(\mathbf{x})>1
\wedge t-s>0,
\end{equation*} 
we can employ Lemma \ref{M.l6.1} for $|\mathcal{B}_{t-s,0,s+p}(\mathbf{x})|$ 
and thus 
\begin{equation*}
|\mathcal{B}_{t,s,p}(\mathbf{x})|
\geq|\mathcal{B}_{t-s,0,s+p}(\mathbf{x})|
>\sum_{i=0}^{t+p}\binom{n+t-s}{i}(q-1)^{i},
\end{equation*}
which completes the proof. 
\end{proof}
%-----------------------------------------------------------------------------------------------%

%-----------------------------------------------------------------------------------------------%

%-----------------------------------------------------------------------------------------------%

%-----------------------------------------------------------------------------------------------%

%-----------------------------------------------------------------------------------------------%


\begin{thebibliography}{100}

\bibitem{Abb.2024}
A. Abbasian, M. Mirmohseni, M. N. Kenari, 
On the size of error ball in DNA storage channels, 2024 [Online]. 
Available: \url{https://arxiv.org/abs/2410.15290}

\bibitem{Abu-Sini.2021}
M. Abu-Sini, E. Yaakobi, 
On Levenshtein’s reconstruction problem under insertions, deletions, and substitutions, 
\textit{IEEE Trans. Inf. Theory}, 67 (11) (2021) 7132--7158. 

\bibitem{Bar-Lev.2023}
D. Bar-Lev, T. Etzion, E. Yaakobi, 
On the size of balls and anticodes of small diameter under the fixed-length Levenshtein metric, 
\textit{IEEE Trans. Inf. Theory}, 69 (4) (2023) 2324--2340. 

\bibitem{Hir.2000}
D. S. Hirschberg, M. Regnier, 
Tight bounds on the number of string subsequences, 
\textit{J. Discrete Algorithms}, 1 (1) (2000) 123--132. 

\bibitem{Lev.1974}
V. I. Levenshtein, 
Elements of coding theory, 
in \textit{Discrete Mathematics and Mathematical Problems of Cybernetics}, 
Nauka, Moscow, pp. 207--305, 1974. 

\bibitem{Lev.1992}
V. I. Levenshtein, 
On perfect codes in deletion and insertion metric, 
\textit{Discrete Math. Appl.}, 2 (3) (1992) 241--258.

\bibitem{Lev.2001a}
V. I. Levenshtein, 
Efficient reconstruction of sequences, 
\textit{IEEE Trans. Inf. Theory}, 47 (1) (2001) 2--22.

\bibitem{Lev.2001b}
V. I. Levenshtein, 
Efficient reconstruction of sequences from their subsequences or supersequences, 
\textit{J. Comb. Theory, Ser. A}, 93 (2) (2001) 310--332.

\bibitem{Lev.2002}
V. I. Levenshtein, 
Bounds for deletion/insertion correcting codes, 
in \textit{Proc. IEEE Int. Symp. Inf. Theory (ISIT)}, p. 370, 2002.

\bibitem{Sal.2013}
F. Sala, L. Dolecek, 
Counting sequences obtained from the synchronization channel, 
in \textit{Proc. IEEE Int. Symp. Inf. Theory (ISIT)}, pp. 2925--2929, 2013.

\bibitem{Sma.2023}
I. Smagloy, L. Welter, A. Wachter-Zeh, E. Yaakobi, 
Single-deletion single-substitution correcting codes, 
\textit{IEEE Trans. Inf. Theory}, 69 (12) (2023) 7659--7671.
\end{thebibliography}
\end{document}